\documentclass{article}
\usepackage{amsmath,amssymb,amsthm}
\usepackage[all]{xy}
\newtheorem{thm}{Theorem}
\newtheorem{cor}[thm]{Corollary}
\newtheorem{prop}[thm]{Proposition}

\newtheorem{rmk}[thm]{Remark}

\def\co{\colon\thinspace}
\newcommand{\mb}[1]{\mathbb{#1}}

\newcommand{\Hom}{\ensuremath{{\rm Hom}}}

\newcommand{\Ind}{\ensuremath{{\rm Ind}}}

\newcommand{\hocolim}{\ensuremath{\mathop{\rm hocolim}}}

\newcommand{\overto}{\mathop\rightarrow}
\newcommand{\into}{\mathop\hookrightarrow}

\newcommand{\Rep}{\ensuremath{\mathop{\rm Rep}}}
\newcommand{\Sum}{\ensuremath{\mathop{\rm Sum}}}
\newcommand{\Map}{\ensuremath{{\rm Map}}}
\newcommand{\GL}{{\rm GL}}
\newcommand{\Uni}{{\rm U}}
\newcommand{\PU}{{\rm PU}}
\newcommand{\Sp}{{\cal S}p}
\newcommand{\Sym}{{\rm Sym}}
\newcommand{\Ad}{{\rm Ad}}
\newcommand{\cB}{{\rm B}}
\newcommand{\cE}{{\rm E}}

\newcommand{\eilm}[1]{\ensuremath{{\rm H} #1}}
\newcommand{\smsh}[1]{\ensuremath{\mathop{\wedge}_{#1}}}

\newcommand{\xym}[1]{
\vskip 0.7pc
\centerline{\xymatrix{#1}}
\vskip 0.7pc
}

\begin{document}
\title{The Bott cofiber sequence in
  deformation $K$-theory and simultaneous similarity in $\Uni(n)$}
\author{Tyler Lawson\footnote{Supported in part by NSF
    grant DMS-0402950.}\\
Department of Mathematics, University of Minnesota\\
email: \texttt{tlawson@math.umn.edu}}

\maketitle

\begin{abstract}
We show that there is a homotopy cofiber sequence of spectra
relating Carlsson's deformation $K$-theory of a group $G$ to its
``deformation representation ring,'' analogous to the Bott
periodicity sequence relating connective $K$-theory to ordinary
homology.  We then apply this to study simultaneous similarity of
unitary matrices.
\end{abstract}

The algebraic $K$-theory of a category uses the machinery of infinite
loop space theory to associate spectra to symmetric monoidal
categories.  The homotopy groups of these spectra give information
about the structure of the category itself.  However, some symmetric
monoidal categories arise with natural topologies on their objects and
morphisms that give information about how objects in the category can
behave in families.

For example, given a group $G$, we can consider the category of its
finite-dimensional complex representations or unitary representations,
each of which comes with a natural topology.  Carlsson's ``deformation
$K$-theory,'' or the associated unitary variant, produces a $K$-theory
spectrum which depends on both the symmetric monoidal structure and
the behavior in families.

The purpose of this article is to identify the cofiber of the Bott map
on unitary deformation $K$-theory (\cite{carlsson}, \cite{tylerprod})
of a finitely generated group $G$.  For a finite group $G$, this
cofiber can be identified with the Eilenberg-MacLane spectrum
associated to the complex representation ring $R[G]$.  More generally
one obtains a ``unitary deformation representation ring,'' also
denoted by $R[G]$, which is a commutative $\mb HZ$-algebra spectrum.
This deformation representation ring was considered in a previous
paper \cite{tylerrep}.  Results of Park and Suh \cite{parksuh} will be
applied to show that this deformation representation ring admits a
cellular construction as an $\eilm{\mb Z}$-module spectrum.

There is a resulting first quadrant Atiyah-Hirzebruch style spectral
sequence converging to the homotopy groups of deformation $K$-theory, as
follows.
\[
E_2^{p,q} = E_3^{p,q} = \pi_p(R[G]) \otimes \pi_q(ku) \Rightarrow
\pi_{p+q} {\cal K}G.
\]

As a side effect of this identification of $R[G]$ with the
cofiber of the Bott map, we obtain results about the homotopy type of
spaces parameterizing representations of the group $G$.  In
particular, when $G$ is free, we obtain information about simultaneous
similarity.

The spectral theorem in linear algebra implies that a unitary matrix
$A$ is determined, up to similarity, by its set of eigenvalues
$\{z_1,\ldots,z_n\}$, counted with multiplicity.  Taking the
eigenvalues of a matrix gives a map from $\Uni(n)$ to the $n$-fold
symmetric product $\Sym^n(S^1)$, inducing a bijection
\[
\Uni(n)^{\Ad}/\Uni(n) \to \Sym^n(S^1).
\]
In fact, both sides have natural topologies that make this map a
homeomorphism.

The {\em simultaneous similarity\/} problem in $\Uni(n)$ is to classify
the orbits of $k$-tuples of matrices $(A_1,\ldots,A_k)$ under unitary
change of basis, or simultaneous conjugation.  There is an analogous
classification in $\GL(n)$ due to Friedland \cite{friedland},
which generalizes the Jordan canonical form but is much more involved.

The simplest invariant that can be extracted from this situation is
the collection of eigenvalues.  This gives a continuous {\em
  eigenvalue map\/}
\[
\phi_{n,k}\co X(n,k) = \left[\Uni(n)^{\Ad}\right]^k/\Uni(n) \to
\left[\Sym^n(S^1)\right]^k.
\]
In addition, there are stabilization maps $X(n,k) \into X(n+1,k)$, given
by
\[
(A_i) \mapsto \left(
\begin{bmatrix}
A_i & 0 \\0  & 1
\end{bmatrix}
\right)
\]
Define $X(\infty,k)$ to be the (homotopy) colimit of the $X(n,k)$.
These stabilization maps commute with the stabilization maps
$\Sym^n(S^1) \into \Sym^{n+1}(S^1)$, given by adding an extra copy of
the basepoint $1$.

We will show that the stable eigenvalue map
\[
\phi_k\co X(\infty,k) \to \left[ \Sym^\infty(S^1)\right]^k
\]
is a homotopy equivalence.  The Dold-Thom theorem already implies that
the map $S^1 \to \Sym^\infty(S^1)$ is an isomorphism on homotopy groups.
Therefore, this result can be rephrased by saying that the map
$(S^1)^k = X(1,k) \to X(\infty,k)$ is a homotopy equivalence.

At the end of this paper we will give two proofs of this result.  The
first proof presented here applies to more general spaces of
representations and makes use of recently developed categories of
module spectra, particularly of smash products over the connective
$K$-theory spectrum.  The author does not know general results about
the stabilization of the homotopy groups of the spaces $X(n,k)$.


In section~\ref{sec:simplicial}, we give an interpretation of the
eigenvalue map in terms of simplicial spaces, using the simplicial
decomposition of $\Uni(n)$ given in Harris \cite{harris}.  We then
establish a geometric proof by establishing contractibility for
various spaces parameterizing multiple hyperplane arrangements in $\mb
C^n$ for large $n$.

The geometric proof amounts to showing that the maps $\phi_k$ are
quasifibrations.  It should be noted that the eigenvalue maps
$\phi_{n,k}$ are not quasifibrations, even for $n=3$, $k \geq 2$.  The
fiber of the eigenvalue map at the basepoint is a single point; an
exercise is to show that the fiber over a point of the form
\[
\left(
\{\zeta_1, \zeta_2, 1,\}, 
\{\zeta_1, \zeta_2, 1,\}, 
\{1,1,1\},
\{1,1,1\},\ldots
\right) \in \left[\Sym^3(S^1)\right]^k,
\]
for $\zeta_1, \zeta_2, 1$ distinct elements of $S^1$, has the homotopy
type of $S^4$.

\section{Definitions}
\label{sec:definitions}

We briefly recall the definition and several properties of deformation
$K$-theory from \cite[Section~4]{tylerprod} and the deformation
representation ring functor from \cite[Section~4]{tylerrep}.

Recall (\cite{segal}, \cite{schwede}) that a $\Gamma$-space $M$ is
a functor from finite based sets to based spaces such that $M(*) = *$.
Associated to a levelwise finite simplicial set $K$, there is an
associated based space $M(K)$ obtained by applying $M$ levelwise and
taking geometric realization.  There is a natural assembly map $K
\smsh{} M(L) \to M(K \smsh{} L)$, and so a $\Gamma$-space gives rise
to a symmetric spectrum
\[
\Sp(M) = \{M(S^n)\}.
\]
If $X$ is a topological abelian monoid, we can define a $\Gamma$-space
associated to $X$ by
\[
X(Z) = F(Z,X),
\]
where $F$ denotes the based mapping space, such that for $\alpha\co Z
\to Z'$, 
\[
\alpha_*(f)(z') = \sum_{\alpha(z) = z'} f(z).
\]

Associated to a (topological) group $G$, we let $\Rep(G)$ be the space
\[
\coprod_{n \in \mb N} \Hom(G, \Uni(n))/\Uni(n),
\]
where the space of homomorphisms has the compact-open topology and
$\Uni(n)$ acts by conjugation.  This space parameterizes isomorphism
classes of unitary representations of $G$.

The operations $\oplus$ and $\otimes$ give rise to the structure of a
commutative topological semiring on $\Rep(G)$.  In particular, the
abelian addition operation $\oplus$ allows us to construct a
spectrum 
\[
R[G] = \Sp\left(\Rep(G)\right).
\]
One can show that $R[G]$ is the spectrum obtained by iterated
application of the the classifying space functor.  This spectrum can
be viewed as a homotopical group completion functor, generalizing the
Grothendieck group construction of the ordinary representation ring.

The natural map
\[
\Rep(G) \to \Omega^\infty R[G]
\]
is a homotopy group completion map.  The operation $\otimes$ gives
rise to the structure of an $E_\infty$-algebra over $\eilm{\mb Z}$ on
$R[G]$.

The construction of $\Rep(G)$ has a $K$-theoretic analogue.  
Let ${\cal U}$ be a fixed countably infinite inner product space over
$\mb C$.  A {\em $G$-plane\/} in ${\cal U}$ is a pair $(V,\rho)$,
where $V$ is a finite dimensional subspace of ${\cal U}$ and $\rho\co
G \to \Uni(V)$ is a group homomorphism.  We define the deformation
$K$-theory of $G$, ${\cal K}G$, to be the $\Gamma$-space given by
\[
{\cal K}G(Z) = \left\{(V_z,\rho_z)_{z \in Z}\ \Big|\ V_z \hbox{ a
    $G$-plane,}\ V_z \perp V_{z'} \hbox{ if } z \neq z',\  V_* = 0\right\}.
\]
For a morphism $\alpha\co Z \to Z'$, we define
\[
\alpha_*((V_z,\rho_z)_{z \in Z}) = \left(\bigoplus_{\alpha(z) = z'} V_z,
\bigoplus_{\alpha(z) = z'} \rho_z\right)_{z' \in Z'}.
\]

Taking isomorphism classes gives a map ${\cal K}G \to
\Rep(G)$ of $\Gamma$-spaces.  Therefore, there is a
natural map of spectra $\Sp({\cal K}G) \to R[G]$.

The justification for the name $K$-theory arises as follows.  We
define the following spaces.
\begin{eqnarray*}
Ob({\cal C}_G) &=& \coprod_{n \in \mb N} \Hom(G, \Uni(n)) \\
Mor({\cal C}_G) &=& \coprod_{n \in \mb N} \Hom(G, \Uni(n)) \times
\Uni(n)
\end{eqnarray*}
View a homomorphism $\rho \in Ob({\cal C}_G)$ as a (unitary)
representation of $G$, and a pair $(\rho, A) \in Mor({\cal C})$ as an
isometry of representations $A\co \rho \to A \rho A^{-1}$.
These form an internal category in spaces; the source, target, unit,
and composition maps are all continuous.  Further, the block
sum $\oplus$ makes this into an internal symmetric monoidal category
(in fact, a permutative category) in spaces.  The spectrum $\Sp({\cal
K}G)$ is homotopy equivalent to the associated $K$-theory object
$K({\cal C}_G)$.

Explicitly, we have a nerve
\[
N({\cal C}_G) \simeq \coprod_{n \in \mb N} \Hom(G,\Uni(n))
\times_{\Uni(n)} \cE\Uni(n).
\]
The permutative category structure makes this into a topological
monoid with an $E_\infty$-H-space structure, and $K({\cal C}_G)$ is the
connective spectrum associated to $N({\cal C}_G)$.  If $G$ is
trivial, the associated spectrum is the connective $K$-theory spectrum.

\section{Filtrations of the representation ring}

We now provide a cellular construction of the topological monoid
$\Rep(G)$ of the previous section.

If $G$ is finitely generated and discrete, the space $\Hom(G,\Uni(n))$
is the set of real points of an algebraic variety, with $\Uni(n)$ acting
algebraically by conjugation.  In particular, by \cite[Theorem
3.7]{parksuh}, it admits the structure of a $\Uni(n)$-CW complex.

For any $N \in \mb N$, let $\Rep(G,N)$ be the submonoid of $\Rep(G)$
generated by the subspace $\coprod_{n \leq N} \Hom(G,\Uni(n))$.  This
gives rise to a sequence of inclusions
\[
* = \Rep(G,0) \subset \Rep(G,1) \subset \Rep(G,2) \subset \cdots
\]

A point of $\Rep(G)$ is an isomorphism class of unitary
representations of $G$.  In particular, any such representation admits
a unique decomposition into irreducible subrepresentations.  Let
$\Sum(G,N) \subset \Hom(G,\Uni(N))$ be the subspace consisting of
those representations which are reducible.

Equivalently, a representation $V \in \Hom(G,\Uni(N))$ is irreducible
if and only if the stabilizer of it under the action of $\Uni(N)$ is
the diagonal subgroup $S^1$, as follows.  If $V \cong V' \oplus V''$, the
stabilizer contains an action of $S^1 \times S^1$ acting individually
on each factor.  Conversely, Schur's lemma shows that the endomorphism
ring of any irreducible object $V$ is a finite dimensional division
algebra over $\mb C$, and hence consists only of scalar maps.  This
shows that the map $\Sum(G,N) \into \Hom(G,\Uni(N))$ must be a
$\Uni(N)$-CW inclusion.

This gives rise to the following diagram of spaces.
\xym{
\Sum(G,N)/\Uni(N) \ar@{^(->}[r] \ar[d] &
\Hom(G,\Uni(N))/\Uni(N) \ar[d] \\
\Rep(G,N-1) \ar[r] &
\Rep(G,N)
}
Applying the free abelian topological monoid functor
$\Sym^\infty((-)_+)$, which is left adjoint to the forgetful functor,
to the top row gives a diagram of abelian topological monoids.
\xym{
\Sym^\infty(\Sum(G,N)/\Uni(N)_+) \ar@{^(->}[r] \ar[d] &
\Sym^\infty(\Hom(G,\Uni(N))/\Uni(N)_+) \ar[d] \\
\Rep(G,N-1) \ar[r] &
\Rep(G,N)
}
The statement that any unitary representation of $G$ is uniquely (up
to isomorphism) a direct sum of irreducible subrepresentations implies
that on the level of underlying abelian monoids, this diagram is a
pushout diagram.

The monoids in the above diagram admit augmentations to the monoid
$\mb N$, and are compact Hausdorff in each fiber.  Therefore, the
above diagram is a pushout diagram of topological abelian monoids.
These pushout diagrams are preserved by cartesian products, and hence
the associated diagram of classifying spaces is a pushout diagram.

We include a proof of the following for completeness.
\begin{prop}
Suppose $A \to B$ is a CW-inclusion and $\Sym^\infty(A_+) \to M$
is a map of topological abelian monoids.  Let $N$ be the pushout of
the diagram
\[
M \leftarrow \Sym^\infty(A_+) \to \Sym^\infty(B_+)
\]
of topological abelian monoids.  Then the map $M \to N$ is a
CW-inclusion, and the sequence of maps
\[
M \to N \to \Sym^\infty(B/A)
\]
induces a homotopy fibration sequence of spectra
\[
\Sp(M) \to \Sp(N) \to \eilm{\mb Z} \smsh{} (B/A).
\]
\end{prop}

\begin{proof}
The pushout $N$ is formed as a sequence of iterated CW attachments
$N_i \to N_{i+1}$, where $N_0 = M$ and
\[
N_{i+1} = N_i \bigcup_{(\cup B^j \times A \times B^{i-j}) / \Sigma_{i+1}}
B^{i+1}/\Sigma_{i+1}.
\]
Weak equivalences are preserved by pushouts along cofibrations, so a
weak equivalence $M' \to M$ of topological monoids induces a homotopy
equivalence of pushouts.  In particular, the natural weak equivalence
\[
B(M,\Sym^\infty(A_+),\Sym^\infty(A_+)) \to M,
\]
using the bar construction with respect to the monoid structure,
induces a weak equivalence of topological abelian monoids
\[
B(M,\Sym^\infty(A_+),\Sym^\infty(B_+)) \to N.
\]
The classifying space functor commutes with products, and hence with
the bar construction.  Upon iterative application, we find that there
is a natural weak equivalence of spectra
\[
B(\Sp(M), \Sp(\Sym^\infty(A_+)), \Sp(\Sym^\infty(B_+))) \to \Sp(N).
\]

The generalized Dold-Thom theorem implies that there is a natural weak
equivalence
\[
\eilm{\mb Z} \smsh{} \Sigma^\infty X_+ \to  \Sp(\Sym^\infty(X_+))
\]
for spaces $X$ of the homotopy type of a CW-complex.  Therefore, there
is a natural weak equivalence
\[
B(\Sp(M), \eilm{\mb Z} \smsh{} \Sigma^\infty A_+, \eilm{\mb Z} \smsh{}
\Sigma^\infty B_+) \to \Sp(N),
\]
where the bar construction is taken with respect to coproduct (wedge)
in spectra.  Equivalently, there is a homotopy pushout diagram of
spectra
\xym{
\eilm{\mb Z} \smsh{} \Sigma^\infty A_+ \ar[r] \ar[d] &
\eilm{\mb Z} \smsh{} \Sigma^\infty B_+ \ar[d] \\
\Sp(M) \ar[r] &
\Sp(N).
}
The result follows by considering the homotopy cofibers of the rows in
this diagram.
\end{proof}

Let $R_N = \Hom(G,\Uni(N)) / \Sum(G,\Uni(N))$; it is a based
$\PU(N)$-CW complex with free action away from the basepoint.

\begin{cor}
\label{cor:fibseq}
The spectrum $R[G]$ is the homotopy colimit of the spectra
\[
\Sp(\Rep(G,N)).
\]
There are fibration sequences of spectra for each $N \geq 1$
\[
\Sp(\Rep(G,N-1)) \to \Sp(\Rep(G,N)) \to \eilm{\mb Z} \smsh{}
(R_N/\PU(N)).
\]
\end{cor}

\begin{proof}
The inclusions $\Rep(G,N-1) \to \Rep(G,N)$ are CW-inclusions and
induce pushout diagrams of spectra.  The space $\Rep(G)$ is therefore
the homotopy colimit of the subspaces $\Rep(G,N)$, and
$\Sp(\Rep(G))$ is the homotopy colimit of its subspectra
$\Sp(\Rep(G,N))$.

The existence of the fibration sequence is immediate from the
proposition.
\end{proof}

\section{Filtrations of unitary deformation $K$-theory}

We briefly recall the following results from \cite{tylerprod}.  From
this point forward we abuse notation by writing ${\cal K}G$ to denote
the spectrum associated to the $\Gamma$-space of
section~\ref{sec:definitions}.

In the previous section we showed that that $\Rep(G)$ has a filtration
by submonoids $\Rep(G,N)$ consisting of representations that are
direct sums of irreducible subrepresentations of dimension $N$ or
smaller.  There is an associated filtration of the symmetric monoidal
category ${\cal C}_G$ by closed subcategories ${\cal C}_{G,N}$, and a
diagram of maps of spectra as follows.
\xym{
K({\cal C}_{G,1}) \ar[r] \ar[d] &
K({\cal C}_{G,2}) \ar[r] \ar[d] & 
\cdots  \ar[r] &
{\cal K}G \ar[d] \\
\Sp(\Rep(G,1)) \ar[r] &
\Sp(\Rep(G,2)) \ar[r] &
\cdots \ar[r] &
R[G].
}

The map $\hocolim K({\cal C}_{G,N}) \to {\cal K}G$ is a weak
equivalence \cite[Proposition~14]{tylerprod}.  Additionally, the
objects $K({\cal C}_{G,N})$ are module spectra over $ku$ for all $N$
\cite[Proposition~30]{tylerprod}.

Recall that $R_N = \Hom(G,\Uni(N)) / \Sum(G,\Uni(N))$.  In
\cite[Section~4]{tylerprod}, for each $N$, a spectrum $ku^{\PU(N)}$
with a continuous action of $\PU(N)$ was constructed, with underlying
spectrum homotopy equivalent to $ku$, so that there is a cofibration
sequence up to homotopy
\[
K({\cal C}_{G,N-1}) \to K({\cal C}_{G,N}) \to
R_N \smsh{\PU(N)} ku^{\PU(N)}.
\]
(\cite[Corollary~19]{tylerprod} and \cite[Corollary~22]{tylerprod}.)
Taking smash products over $ku$ with $\eilm{\mb Z}$ gives a natural
cofibration sequence
\[
\eilm{\mb Z} \smsh{ku} K({\cal C}_{G,N-1}) \to
\eilm{\mb Z} \smsh{ku} K({\cal C}_{G,N}) \to
\eilm{\mb Z} \smsh{} (R_N/\PU(N)).
\]
(\cite[Section~8]{tylerprod}.)  

\begin{thm}
\label{thm:equiv}
The map $K({\cal C}_{G}) \to R[G]$ induces a weak
equivalence
\[
\eilm{\mb Z} \smsh{ku} K({\cal C}_{G}) \to R[G].
\]
\end{thm}

\begin{proof}
The $ku$-module structure on $K({\cal C}_{G,N})$ is induced by tensor
product with trivial vector spaces.  It coherently commutes with the
abelian group structure on $\Rep(G,N)$ via the augmentation map
sending a vector space to its dimension.  Therefore, the map $N({\cal
C}_{G,N}) \to \Rep(G,N)$ induces a map of $ku$-modules.

The proof proceeds by proving inductively that the adjoint map of
$\eilm{\mb Z}$-modules
\[
\eilm{\mb Z} \smsh{ku} K({\cal C}_{G,N}) \to \Sp(\Rep(G,N))
\]
is a weak equivalence, and taking homotopy colimits.  By the
five-lemma, it suffices to show that the induced maps of homotopy
cofibers are weak equivalences for all $N$.  By
corollary~\ref{cor:fibseq} and \cite[Section~8]{tylerprod}, these
homotopy cofibers are both weakly equivalent to
\[
\eilm{\mb Z} \smsh{} (R_N/\PU(N)).
\]
Therefore, it suffices to produce a map demonstrating that this map is
a weak equivalence.  There is a natural diagram of maps of spaces
\xym{
\Sum(G,\Uni(N)) \ar[r] \ar[d] &
N({\cal C}_{G,N-1}) \ar[d] \ar[r] &
\Rep(G,N-1) \ar[d] \\
\Hom(G,\Uni(N)) \ar[r] &
N({\cal C}_{G,N}) \ar[r] &
\Rep(G,N).
}
\noindent Suspension is left adjoint to the forgetful functor to
spaces, so there is an induced diagram of maps of symmetric spectra
\xym{
\Sigma^\infty \Sum(G,\Uni(N))_+ \ar[r] \ar[d] &
K({\cal C}_{G,N-1}) \ar[r] \ar[d] &
\Sp(\Rep(G,N-1)) \ar[d] \\
\Sigma^\infty \Hom(G,\Uni(N))_+ \ar[r] &
K({\cal C}_{G,N}) \ar[r] &
\Sp(\Rep(G,N)).\\
}
Taking pushouts in columns gives maps
\[
\Sigma^\infty R_N \to ku^{\PU(N)} \smsh{\PU(N)} R_N \to \eilm{\mb Z}
\smsh{} (R_N/\PU(N))
\]
whose adjoint maps
\[
\eilm{\mb Z} \smsh{\PU(N)} R_N \to \eilm{\mb Z} \smsh{ku} ku^{\PU(N)}
\smsh{\PU(N)} R_N \to \eilm{\mb Z} \smsh{} (R_N/\PU(N))
\]
are equivalences.
\end{proof}

\begin{cor}
\label{cor:ahss}
There is a homotopy cofiber sequence of $ku$-modules
\[
\Sigma^2 {\cal K}G \overto^\beta {\cal K}G \to R[G],
\]
where $\beta$ is multiplication by the Bott element in $\pi_2(ku)$.
There is a corresponding convergent ``Atiyah-Hirzebruch'' spectral
sequence with $E_2$-term
\[
\pi_p(R[G]) \otimes \pi_q(ku) \Rightarrow \pi_{p+q} {\cal K}G.
\]
\end{cor}

\begin{proof}
This follows by smashing the homotopy cofiber sequence
\[
\Sigma^2 ku \overto^\beta ku \to \eilm{\mb Z}
\]
with the spectrum ${\cal K}G$ over $ku$, and using the theorem to
identify the terms in the result.  The spectral sequence follows by
considering the tower of spectra
\[
\cdots \to \Sigma^4 {\cal K}G \to \Sigma^2 {\cal K}G \to {\cal K}G,
\]
whose filtration quotients are $\Sigma^{2k} R[G]$.
\end{proof}

\section{Example computations}

In this section, we analyze irreducible representations to compute the
deformation ring spectrum $R[G]$, and then apply
Corollary~\ref{cor:ahss} to obtain information about the deformation
$K$-theory groups of several groups.

For further examples relating deformation $K$-theory of surface groups
to gauge theory, the reader should consult \cite{ramrasyangmills}.

\subsection{Finitely generated abelian groups}

Let $G$ be a finitely generated abelian group, with
character group $G^* = \Hom(G,\Uni(1))$.  Any irreducible
representation is uniquely, up to isomorphism, a direct sum of
characters.  The topological monoid $\Rep(G)$ is the infinite
symmetric product $\Sym^\infty(G^*)$, and so $R[G] \simeq \eilm{\mb Z}
\smsh{} G^*.$

In particular, if $G \cong \mb Z^r \oplus A$ where $A$ is finite, then
$G^* \cong (S^1)^r \times A^*$, and so we obtain the following.
\[
\pi_* R[G] \cong \bigoplus_{a \in A^*} H_*(S^1)^{\otimes r} \cong
H_*(S^1)^{\otimes r} \otimes R[A]
\]
In particular, it is free abelian in each degree.  The $E_2$-term of
the spectral sequence for deformation $K$-theory is therefore
\[
H_*(S^1)^{\otimes r} \otimes R[A] \otimes \mb Z[\beta].
\]

It remains to exclude the possibility of differentials in this
spectral sequence.  Either naturality in $G$ or the results of
\cite{tylerprod} imply that the spectral sequence degenerates at the
$E_2$-term.

\subsection{The integer Heisenberg group}

Let $G$ be the integer Heisenberg group of upper triangular integer
matrices with $1$ on the diagonal.  In \cite{tylerrep}, the
deformation representation ring $R[G]$ was shown to satisfy
\[
\pi_* R[G] \cong
\begin{cases}
\oplus\,\mb Z & \text{if }* = 0,\\
\oplus\,\mb Z^2 & \text{if }* = 1,\\
\oplus\,\mb Z & \text{if }* = 2,\\
0 & \text{otherwise}.
\end{cases}
\]
Here the direct sum ranges over roots of unity in $\mb C$; these index
the irreducible representations via the ``central character.''  The
spectral sequence for deformation $K$-theory is therefore forced to
degenerate at $E_2$, with no hidden extensions possible as all groups
involved are free.  Therefore, we find that 
\[
\pi_* {\cal K}(G) \cong
\begin{cases}
\oplus\,\mb Z & \text{if }* = 0,\\
\oplus\,\mb Z^2 & \text{if }* \geq 1.\\
\end{cases}
\]

\subsection{$\mb Z \rtimes \mb Z/2$}
Let $G$ be the semidirect product $\mb Z \rtimes \mb
Z/2$, where $\mb Z/2$ acts by negation on $\mb Z$.  It has an
abelian subgroup $\mb Z$ of index $2$, and hence any irreducible
representation has dimension $1$ or $2$.

The commutator subgroup of $G$ is $2 \mb Z$, and there are four
1-dimensional representations $1, \sigma, \tau, \sigma\tau$.  (Here we
take $\sigma$ to be the nontrivial representation factoring through
the ``obvious'' quotient map $\mb Z \rtimes \mb Z/2 \to \mb
Z/2$.)  Therefore, $R_1 / \PU(1) \cong \vee^4 S^0$.

For any $\alpha \in S^1$, there is a corresponding unitary character
of $\mb Z$ (also denoted by $\alpha$) which sends the generator to
$\alpha$.  The induced representation $V_\alpha = \Ind_{\mb
Z}^G(\alpha)$ is a two-dimensional unitary representation whose
restriction to $\mb Z$ is isomorphic to $\alpha \oplus \alpha^{-1}$.
One readily checks the following facts.
\begin{itemize}
\item $V_\alpha \cong V_\beta$ if and only if $\alpha = \beta^{\pm
    1}$.
\item $V_{\alpha}$ is irreducible if and only if $\alpha \neq \pm 1$.
\item $V_1 \cong 1 \oplus \sigma$ and $V_{-1} \cong \tau \oplus \sigma \tau$.
\item All 2-dimensional representations of $G$ are either reducible or
  isomorphic to $V_\alpha$ for some $\alpha$.
\end{itemize}

As a result, the space $R_2 / \PU(2)$ of 2-dimensional representations
modulo reducibles is homeomorphic to $[0,1]/\partial [0,1] \simeq
S^1$.  The cofiber sequences of corollary~\ref{cor:fibseq} degenerate
to a single cofiber sequence
\[
\eilm{\mb Z} \smsh{} (\vee^4 S^0) \to R[G] \to \eilm{\mb Z} \smsh{}
([0,1] / \partial [0,1]).
\]
Therefore, $\pi_* R[G] = 0$ for $* > 0$, and there is a short exact
sequence as follows.
\xym{
0 \ar[r] &
\mb Z \ar[r] &
\mb Z \oplus \mb Z\sigma \oplus \mb Z\tau \oplus \mb Z \sigma \tau
\ar[r] &
\pi_0 R[G] \ar[r] &
0.
}
\noindent The left-hand map in this sequence is multiplication by $(1 +
\sigma) - (\tau + \sigma \tau)$.

As the homotopy of $R[G]$ is concentrated in degree zero, the spectral
sequence for the deformation $K$-theory degenerates and we find
\[
\pi_* {\cal K}(G) \cong
\begin{cases}
\mb Z^3 & \text{if }*\text{ is even, } * \geq 0,\\
0 & \text{otherwise.}
\end{cases}
\]
(The degeneration of the spectral sequence actually implies that the
homotopy type of the spectrum is $\vee^3 ku$.)

We note that this group is isomorphic to the amalgamated product $\mb
Z/2 \ast \mb Z/2$, and the main theorem of \cite{ramras} recovers this
result as part of a general formula for amalgamated products.

\subsection{$\mb Z^2 \rtimes \mb Z/4$}

We list one final example which is not known by methods of excision or
product formulas.

Suppose $G$ is the semidirect product $\mb Z^2 \rtimes \mb Z/4$, where
the cyclic group of order 4 acts on $\mb Z^2$ by the matrix
\[\begin{bmatrix}0 & -1 \\1 & 0\end{bmatrix}.\]
Choose generators $x$ and $y$ for $\mb Z^2$.  The group $G$ has an
index $4$ abelian subgroup, and so the irreducible representations
have dimensions $1$, $2$, or $4$.

More specifically, let $T = \Hom(\mb Z^2, \Uni(1))$ be the character
group of $\mb Z^2$, with action of $\mb Z/4$ by precomposition.
Elementary Frobenius reciprocity breaks the irreducible
representations of $G$ into the following types.
\begin{itemize}
\item Associated to each character in $T$ fixed by $\mb Z/4$, there
  are $4$ distinct extensions to irreducible $1$-dimensional
  representations.  These are acted on freely transitively by the
  character group of $\mb Z/4$.

There are precisely $2$ characters in $T$ fixed by $\mb Z/4$, given by
the trivial character and the character $x \mapsto -1, y \mapsto -1$.
The group of characters of $G$ is $(\mb Z/2 \times \mb Z/4)^*$.
\item Associated to any $\mb Z/4$-orbit in $T$ of order 2, each
  representative has $2$ distinct extensions to $1$-dimensional
  representations of $\mb Z^2 \rtimes \mb Z/2$.  These induce to 2
  distinct irreducible $2$-dimensional representations of $G$
  determined only by the orbit.  These are interchanged by the
  character group of $\mb Z/4$. 

There is precisely $1$ orbit in $T$ of size 2, with a representative 
given by the character $x \mapsto -1, y \mapsto 1$.  There are then $2$
irreducible representations of degree 2.
\item Associated to each $\mb Z/4$-orbit in $T$ of order 4, any
  representative in the orbit induces to an irreducible $4$-dimensional
  representation of $G$.  This is fixed by the character group of $\mb
  Z/4$.

  The space of isomorphism classes of representations of degree 4,
  modulo reducibles, is therefore the quotient of $T$ by $\mb Z/4$
  (homeomorphic to $S^2$), modulo the 3 points corresponding to orbits
  of size less than $4$.  We can give this space a cell structure with
  two $1$-dimensional cells (attaching the three points reducing to
  the basepoint), together with a $2$-cell attached via a map trivial
  in homology.
\end{itemize}

One can carry out analysis as in the previous example to show that the
boundary map on the generators of the $1$-cells in homology injects to
a direct summand onto the previously attached $0$-cells.  One can then
determine that the deformation representation ring $R[G]$ has homotopy
as follows.
\[
\pi_* R[G] \cong
\begin{cases}
\mb Z^8 & \text{if }* = 0,\\
\mb Z & \text{if }* = 2,\\
0 & \text{otherwise.}
\end{cases}
\]
The spectral sequence for the deformation $K$-theory then degenerates
at the $E_2$-term, and we find
\[
\pi_* {\cal K}(G) \cong
\begin{cases}
\mb Z^8 & \text{if }* = 0,\\
\mb Z^9 & \text{if }*\text{ is even, } * \geq 0,\\
0 & \text{otherwise.}
\end{cases}
\]
(The degeneration of the spectral sequence actually implies that the
homotopy type of the spectrum is $(\vee^8 ku) \vee \Sigma^2 ku$.)

\section{Representation ring spectra of free groups}

Let $F_k$ be the free group on $k$ generators $x_1,\ldots,x_k$ with
the discrete topology.  A unitary representation of $F_k$ consists of
a choice of image of each generator.  Therefore, 
\[
\Rep(F_k) = \coprod_{n \in \mb N} \left(\Uni(n)^\Ad\right)^k/\Uni(n) =
  \coprod_{n \in \mb N} X(n,k).
\]
The direct sum maps $X(n,k) \times X(m,k) \to X(m+n,k)$ respect
stabilization, and therefore give rise to the structure of an
abelian topological monoid on $X(\infty,k)$.  In particular, there
is a map of abelian topological monoids
\[
\Rep(F_k) \to \mb Z \times X(\infty,k).
\]
The spaces $X(n,k)$ are connected and have abelian fundamental group,
so the same holds for $X(\infty,k)$.

Due to classical results of Quillen \cite[Appendix~Q]{quillen}, the
homotopy group completion map $\Rep(G) \to \Omega \cB \Rep(G)$
is characterized as inducing a localization map on homology
\[
H_*(\Rep(G)) \to (\pi_0 \Rep(G))^{-1} H_*(\Rep(G)).
\]
In particular, the map $\Rep(F_k) \to \mb Z \times X(\infty,k)$ is a
homotopy group completion map.  As a result, we find that
\[
\pi_*(R[G]) \cong
\begin{cases}
\mb Z &\hbox{if }* = 0\\
\pi_*(X(\infty,k)) &\hbox{if }* > 0\\
\end{cases}
\]

\begin{rmk}
Ideally, one would like to prove a stability result at this point.
This is not strictly necessary to show that the given map is a group
completion; this is unnecessary for topological monoids which are
homotopy commutative, since these admit a ``calculus of fractions''.
See \cite[Appendix Q]{quillen}.
\end{rmk}

\begin{prop}
As a spectrum,
\[
R[F_k] \simeq \eilm{\mb Z} \vee \left(\vee^k \Sigma
   \eilm{\mb Z}\right).
\]
The group $\pi_1(R[F_k])$ is isomorphic to $\Hom(F_k,\mb Z)$,
naturally in maps of free groups.
\end{prop}

\begin{proof}
The deformation $K$-theory spectrum of  $F_k$ is the spectrum
associated to the $E_\infty$-$H$-space
\[
\coprod_{n \in \mb N} \left[\Uni(n)^\Ad\right]^k \times_{\Uni(n)}
\cE\Uni(n).
\]
We briefly sketch an identification of the homotopy type of this
spectrum; a more general decomposition of the homotopy type for free
products can be found in \cite{ramras}.

Recall that for a group $G$, there is a natural weak equivalence
\[
G^\Ad \times_G EG \simeq \Lambda BG,
\]
the free loop space on $BG$ \cite[Proposition~2.6]{bhm}.  Taking the
$k$-fold fiber product over $BG$, we find that there is a natural weak
equivalence
\[
\left[G^\Ad\right]^k \times_G EG \simeq \Map(BF_k, BG).
\]
Naturality implies that the weak equivalence
\[
\coprod_{n \in \mb N} \left[\Uni(n)^\Ad\right]^k \times_{\Uni(n)}
\cE\Uni(n) \simeq \Map\left(BF_k, \coprod_{n
  \in \mb N} \cB\Uni(n)\right)
\]
respects the $E_\infty$-structure, where the $E_\infty$-structure on
the right-hand space is derived from the range of the mapping space.

The spectrum associated to $\coprod \cB\Uni(n)$ is the connective
$K$-theory spectrum $ku$.  Therefore, there are maps
\[
\Map\left(BF_k, \coprod \cB\Uni(n)\right) \to \Map\left(BF_k,
\Omega^\infty ku\right) \simeq \Omega^\infty F\left(\Sigma^\infty
(BF_k)_+, ku\right).
\]
This composite map is a homotopy group completion map by inspection.
Therefore, the map
\[
{\cal K}F_k \to F\left(\Sigma^\infty(BF_k)_+, ku\right)
\]
is an isomorphism on homotopy groups in positive degrees, and the
left-hand spectrum is connective.  The function spectrum is
equivalent to
\[
ku \vee \left(\bigvee^k \Omega ku\right)
\]
as a $ku$-module.  The connective cover of this spectrum is
\[
{\cal K}F_k \simeq ku \vee \left(\bigvee^k \Sigma ku\right)
\]
by Bott periodicity.

By theorem~\ref{thm:equiv}, we then find that
\[
R[F_k] \simeq \eilm{\mb Z} \vee \left(\bigvee^k \Sigma \eilm{\mb Z}\right).
\]
The first homotopy group of $R[F_k]$ has a natural isomorphism to the
$(-1)$'st homotopy group of $F((BF_k)_+,ku)$ by Bott
periodicity, which gives rise to the natural isomorphism
\[
\pi_1 R[F_k] \cong \Hom(F_k, \mb Z).
\]
\end{proof}

\begin{cor}
The eigenvalue map $X(\infty,k) \to \left[\Sym^\infty(S^1)\right]^k$
is a weak equivalence; in particular, the higher homotopy groups of
$X(\infty,k)$ vanish.
\end{cor}

\begin{proof}
This follows from identification of the identity component of
$\Omega^\infty R[F_k]$ with $X(\infty,k)$, and the eigenvalue map with
the product of the restriction maps
\[
X(\infty,k) \to X(\infty,1)^k,
\]
which is an isomorphism on $\pi_1$.
\end{proof}

\section{Simplicial interpretation of the eigenvalue map}
\label{sec:simplicial}

In \cite{harris}, the spectral theorem was be reinterpreted as a
simplicial decomposition of the conjugation action of $\Uni(n)$ on
itself.  We will now recall this construction.

For $\{n_i\}_{i=1}^p$ a sequence of integers with $\sum n_i \leq n$,
define 
\[
{\rm Gr}(n_1,n_2,\cdots;n)= 
\Uni(n)\Big/ \left[\Uni(n_1)\times\Uni(n_2) \times \cdots \times \Uni(n -
\Sigma n_i)\right].
\]
This space is a Grassmannian parameterizing configurations
$(V_1,V_2,\cdots)$ of orthogonal systems of subspaces in $\mb C^n$
with ${\rm dim}(V_i) = n_i$.  It has a natural left action of $\Uni(n)$.

Define a simplicial space $X_.$ by
\[
X_p = \coprod_{\sum_{i=1}^p n_i \leq n} {\rm Gr}(n_1,n_2,\cdots;n).
\]
Face maps are given as follows:
\[
d_i(V_1,\cdots,V_p) = 
\begin{cases}
(V_2,\cdots,V_p) & \hbox{if }i = 0,\\
(V_1,\cdots,V_i + V_{i+1}, \cdots,V_p) & \hbox{if }0 < i < p,\\
(V_1,\cdots,V_{p-1}) & \hbox{if }i = p.\\
\end{cases}
\]
The degeneracy maps are insertion of a zero-dimensional subspace.

A point of the geometric realization $|X_.|$ consists of an
arrangement $(V_1,\cdots,V_p)$ of orthogonal hyperplanes and a point
of $\Delta^p$, i.e. a sequence of numbers $0 \leq t_1 \leq \cdots \leq
t_p \leq 1$.  Define a map $f\co |X_.| \to \Uni(n)$ by sending this
point to the matrix $A$ such that each space $V_i$ is an eigenspace
for $A$ with eigenvalue $e^{2\pi i t_i}$, and the orthogonal
complement of $\oplus V_i$ is acted on trivially by $A$.  The map $f$
is a homeomorphism of $\Uni(n)$-spaces.

There is a map of simplicial spaces
\[
X_. = \left\{\coprod_{\sum n_i \leq n} {\rm Gr}(n_1,\cdots;n)\right\} \to
\left\{\coprod_{\sum n_i \leq n} *\right\}.
\]
The right-hand space is obtained from $X_.$ by taking the quotient by
the action of $\Uni(n)$.  The right-hand simplicial set is
$\Sym^n(S^1)$.

By taking $k$-fold products, we recover the map of spaces
\[
\Uni(n)^k \cong \left|X_.^k\right| \to
\left[\Sym^n(S^1)\right]^k.
\]

By taking quotients by the conjugation action, we find that the
map $X(n,k) \to \left[\Sym^n(S^1)\right]^k$ can be
expressed as the geometric realization of the following map of
simplicial spaces.
\[
\left\{\coprod_{n_{i,j}} 
\left[\Uni(n) \Big\backslash \prod_{i=1}^k {\rm
    Gr}(n_{i,1},n_{i,2},\cdots;n)\right]
\right\} \to \left\{\coprod_{n_{i,j}} *\right\}
\]

Therefore, one way to get estimates on the connectivity of the
eigenvalue map would be to obtain increasing bounds on the
connectivity of the spaces on the left.

Note that the space
\[
\Uni(n) \Big\backslash \prod_{i=1}^k {\rm Gr}(n_{i,1},n_{i,2},\cdots;n)
\]
becomes fixed for $n \geq N = \sum_{i,j} n_{i,j}$; this follows because
any configuration of hyperplanes of this type is contained within its
span, which is of dimension less than or equal to $N$.  

\begin{prop}
The space 
\[
Y_n = \Uni(n) \Big\backslash \prod_{i=1}^k {\rm Gr}(n_{i,1},n_{i,2},\cdots;n)
\]
is contractible for large $n$.
\end{prop}

\begin{proof}
This sequence of spaces stabilizes for large $n$, so it suffices to
show that the stabilization map
\[
s\co Y_n \to Y_{kn}
\]
is null-homotopic.

Write $\mb C^{kn} \cong V^{\oplus k}$, where $V = C^n$.  For $2 \leq i
\leq k$ and $0 \leq \theta \leq \pi/4$, let $A_i(\theta)$ be the block
element of $\Uni(kn)$
\[
\begin{bmatrix}
\cos \theta I & 0 & & \sin \theta I & \\
0 & I & \cdots & 0 & \cdots \\
& \vdots & & \vdots & \\
-\sin \theta I & 0 & \cdots & \cos \theta I & \cdots \\
& \vdots & & \vdots & \\
\end{bmatrix},
\]
which rotates the first copy of $V$ to the $i$th copy, leaving the
other copies fixed.  (For simplicity, we define $A_1(\theta)$ to be
the identity.)

One then checks that we have a well-defined homotopy
\[
(H_1,H_2,\ldots,H_k,\theta) \mapsto (A_1(\theta) s(H_1), A_2(\theta)
s(H_2), \ldots, A_k(\theta s(H_k)))
\]
from the stabilization map $s$ to the map
\[
(H_1, H_2, \ldots, H_k) \mapsto (H_1 \oplus 0, 0 \oplus H_2 \oplus 0,
\ldots, 0 \oplus H_k).
\]
However, the right-hand side is constant after the quotient by the
action of $\Uni(kn)$.
\end{proof}

We find that the stable eigenvalue map of the introduction is a
homotopy equivalence from this proposition and the simplicial
decomposition of the stable eigenvalue map.

Stability questions naturally give rise to the following question: How
does the connectivity of these spaces of hyperplane arrangements
depend on the $n_{i,j}$ and $n$?

One can obtain some partial answers to this question.  For example,
\[
\Uni(n) \Big\backslash \left[{\rm Gr}(n_1;n) \times {\rm Gr}(n_2;n)\right]
\]
is always contractible.  Given an $n_1$-dimensional plane $V$ and an
$n_2$-dimensional plane $W$ in $\mb C^n$, let $p$ be the orthogonal
projection from $W$ to $V$ and $q = p^T$ the projection from $V$ to
$W$.  The singular value decomposition in linear algebra shows that
this configuration is determined up to isomorphism by the eigenvalues
$\sigma_1^2 \geq \sigma_2^2 \geq \ldots$ of $pq$, which agree with
those of $qp$ up to additional zeros.  (This method was indicated to
us by Neil Strickland.)

\nocite{*}
\bibliography{ahss}

\end{document}